\documentclass[11pt]{amsart}

\usepackage{setspace}
\usepackage{amsthm}

\usepackage{graphicx}

\theoremstyle{plain}
\newtheorem{theorem}{Theorem}
\newtheorem{lemma}[theorem]{Lemma}

\theoremstyle{definition}
\newtheorem{definition}[theorem]{Definition}

\usepackage{times}
\usepackage[T1]{fontenc}
\usepackage[latin1]{inputenc}
\usepackage{amssymb}
\usepackage{amsmath}
\usepackage{algorithm}
\usepackage{algorithmic}

\begin{document}

\title[Perturbed preconditioned inverse iteration for operator EVP]{Perturbed preconditioned inverse iteration for operator eigenvalue problems with applications to adaptive wavelet discretization}

\author[T. Rohwedder]{Thorsten Rohwedder}
\address{Sekretariat MA 5-3\\ Institut f\"ur Mathematik\\ TU Berlin\\ Stra{\ss}e des 17. Juni 136 \\ 10623 Berlin, Germany}
\email{rohwedde@math.tu-berlin.de}
\urladdr{http://www.math.tu-berlin.de/~rohwedde/}

\author[R. Schneider]{Reinhold Schneider}
\email{schneidr@math.tu-berlin.de}
\urladdr{http://www.math.tu-berlin.de/~schneidr/}

\author[A. Zeiser]{Andreas Zeiser}
\address{Sekretariat MA 3-3\\ Institut f\"ur Mathematik\\ TU Berlin\\ Stra{\ss}e des 17. Juni 136 \\ 10623 Berlin, Germany}
\email{zeiser@math.tu-berlin.de}
\urladdr{http://www.math.tu-berlin.de/~zeiser/}



\date{July 31, 2009}

\begin{abstract}
  In this paper we discuss an abstract iteration scheme for the calculation of 
  the smallest eigenvalue of an elliptic operator eigenvalue problem. 
  A short and geometric proof based on the preconditioned inverse iteration (PINVIT) 
  for matrices [Knyazev and Neymeyr, (2009)] is extended to the case of operators.  
  We show that convergence is retained up to any tolerance if one only uses approximate 
  applications of operators which leads to the perturbed preconditioned inverse iteration (PPINVIT). 
  We then analyze the Besov regularity of the
  eigenfunctions of the Poisson eigenvalue problem on a polygonal domain, 
  showing the advantage of an adaptive solver to uniform refinement when using a stable wavelet base.
  A numerical example for PPINVIT, applied to the model problem on the L-shaped domain, is shown to reproduce the predicted behaviour.
\end{abstract}

\maketitle

\section{Introduction}
\label{sec:introduction}

 In problems arising from physics and engineering one is interested in finding the smallest eigenvalue and/or corresponding eigenfunction of a given elliptic partial differential equation. Depending on the context, this can be for example the lowest vibrational mode in mechanics, or the ground state energy in chemical structure calculation.

In a standard way an eigenvalue problem is posed in a weak formulation \cite{bo91}. We are looking for the smallest eigenvalue $\lambda \in \mathbb{R}$ and corresponding eigenvector $u \in V$, such that
\begin{eqnarray}
  \label{eq:evp_bilinear}
  a(u,v) = \lambda (u,v), \quad \textnormal{for all } v \in V,
\end{eqnarray}
where $V$ is an appropriate Banach space (e.g. $H^1$) that is a dense and continuously embedded subspace of a Hilbert space $H$ (e.g. $L_2$) with inner product $(\cdot, \cdot)$. We assume that $a$ is a bounded, symmetric and strongly positive bilinear form. Furthermore we assume that the smallest eigenvalue $\lambda_1$ is simple and well separated from the rest of the spectrum.

Using Finite Element Methods (FEM), the eigenvalue problem can be efficiently solved numerically. However when the eigenfunction exhibits singularities, one has to use adaptive strategies to retain efficiency. For the iterative mesh generation one often uses local error estimators or indicators \cite{neymeyr02} or, more recently, also dual weighted residual based goal oriented error estimators \cite{br01,hr01}. In practice these methods perform well, but optimal convergence rates cannot be proven yet \cite{gg07}.

A benchmark for optimal convergence rate is the nonlinear best $N$-term approximation
of the solution \cite{devore98}. Therefore one can expect an adaptive algorithm at best to 
calculate an approximation to the solution with an effort which is proportional to the
degrees of freedom needed for a best $N$-term approximation of the same accuracy.

In this sense, a recent article \cite{drsz08} following the spirit of
\cite{cdd01} showed optimal convergence of a perturbed preconditioned 
inverse operation for the solution of elliptic eigenvalue problems . As a basis one uses the operator formulation 
\begin{eqnarray}
  \label{eq:evp_operator}
  A u = \lambda E u,
\end{eqnarray}
where $A$ corresponds to the bilinear form $a$ of equation (\ref{eq:evp_bilinear}) 
and $E$ results from the $H$-inner product. Then given a preconditioner $P$ for $A$ one determines
\begin{eqnarray}
  \label{eq:evp_pinvit}
  v^{n+1} = v^n +  \alpha P^{-1}(A v^n - \mu(v^n) E v^n)
\end{eqnarray}
up to an accuracy $\varepsilon^n$ in each step, where $\alpha$ is an appropriate step length 
and $\mu$ is the Rayleigh quotient.  The proof of convergence uses the fact that
in a neighborhood of the eigenfunction the iteration is contracting for the part perpendicular
to the corresponding eigenspace.

In the case of matrices a more geometrical proof is known \cite{kn08}. This proof assures
convergence to the smallest eigenvalue for all starting vectors whose Rayleigh quotient 
lies between the first and the second eigenvalue. Therefore the domain of convergence
can be substantially bigger compared to the alternative proof.
The first aim is therefore to extend the proof from the matrix case \cite{kn08} to 
abstract spaces. This will substantially shorten our proof and improve our result from \cite{drsz08}. 
Moreover a more geometric and intuitive interpretation of the iteration is possible. 
We will show that, in order to retain convergence, each operator application has to be performed only with an accuracy proportional to the
current error in the eigenfunction. We will show how the idealized iteration can be performed using only approximate
operator applications which is a common practice in adaptive wavelet methods \cite{cdd01,drsz08}.

Our second aim is to apply the present abstract iteration to wavelet discretization.
We will apply the adaptive wavelet algorithm to the  model problem of a planar
Poisson eigenvalue problem on a polygonal domain. For these ansatz spaces the rate of approximation
of an eigenfunction is determined by the regularity of the function in terms of
Besov spaces. Therefore we will first determine this kind of regularity for the eigenfunctions. It will
be shown that the eigenfunctions can be approximated arbitrarily well 
provided that the wavelets have a sufficient number of vanishing moments. 
This is in contrast to Sobolev regularity, where the biggest inner angle of the domain restricts
the smoothness of the eigenfunctions. Therefore, for domains with reentrant corners,
the adaptive scheme is superior to uniform refinement. 
We conclude the practical part by providing some numerical results for the L-shaped domain.

We will proceed along the following line. First, in Section 2, we fix notation and
will rewrite the eigenvalue problem in terms of operators. After that the
convergence of the abstract iteration including perturbations will be shown in Section 3.
In Section 4 we will concentrate on perturbations resulting from inexact operator 
applications. In the last section we will apply the abstract iteration to the case of
a planar Poisson eigenvalue problem, calculate the regularity of the 
eigenfunctions and provide numerical results for the L-shaped domain.

\section{Operator formulation}
\label{sec:operator_formulation}

In this section we will introduce the notation, state the basic assumptions and pose the problem in terms of operators. This is done using the abstract setting of a Gelfand triple which will simplify the later analysis. 

For that purpose let $(H,(\cdot,\cdot),|\cdot|)$ be a separable Euclidean Hilbert space, and $(V,\|\cdot\|)$ a reflexive and separable Banach space such that $V\subset H$ is dense and continuously embedded in $H$, i.e. 
\begin{eqnarray}
  \label{eq:norm_constant}
  |v| \le \alpha \|v\| \quad \textnormal{ for all } v\in V.
\end{eqnarray}
Denote by $(H^*,|\cdot|_*)$ and $(V^*, \|\cdot\|_*)$ the respective dual spaces of $H$ and $V$. The dual pairing on $V^*$ and $V$ is given by $\langle \cdot, \cdot \rangle: V^* \times V   \rightarrow  \mathbb{R}$.
The spaces $V\subset H \cong H^* \subset V^*$ form a Gelfand triple  by identifying $H^*$ and $H$ by the Riesz representation theorem. 

Assume that we are given a bilinear form $a: V \times V \rightarrow \mathbb{R}$ which is bounded, symmetric and strongly positive. We will consider the problem of finding the smallest eigenvalue and corresponding eigenvector of the weak eigenvalue problem
\begin{eqnarray*}
  a(u,v) = \lambda (u,v) \quad \textnormal{for all } v \in V.
\end{eqnarray*}
Equivalently this equation can also be written in operator form. Through the Riesz representation theorem, the bilinear form $a$ uniquely determines an operator $A: V \rightarrow V^*$ satisfying
\begin{eqnarray*}
  a(u,v) = \langle A u,v\rangle \quad \textnormal{for all } u,v \in V.
\end{eqnarray*}
$A$ is bounded, strongly positive, and symmetric with respect to the dual pairing $\langle \cdot, \cdot \rangle$ in the sense that
\begin{eqnarray*}
  \langle A v, u \rangle = \langle A u, v \rangle, \quad \textnormal{for all } v,u \in V.
\end{eqnarray*}
Hence there exist constants $\sigma_0$ and $\sigma_1$ such that 
\begin{eqnarray}
  \label{eq:A_constants}
  \sigma_0 \|v\|^2 \le \langle A v, v\rangle \le \sigma_1 \|v\|^2 \quad
  \textnormal{for all } v \in V.
\end{eqnarray}
For the formulation of the eigenvalue problem in terms of operators we introduce the mapping
\begin{eqnarray*}
  E: H \rightarrow H^*, \quad v \mapsto (\cdot, v)
\end{eqnarray*}
which is induced by the inner product $(\cdot,\cdot)$ on $H$. For convenience we will also denote its restriction $\left. E \right|_V \in \mathcal{L}(V,V^*)$ by $E$. 

Now an equivalent definition of a weak eigenvalue in terms of operators can be made.
\begin{definition}
  \label{def:eigenvalue}
  Let $A:V\rightarrow V^*$ be a symmetric, bounded and strongly positive operator. $\lambda \in \mathbb{R}$ is a \emph{(weak) eigenvalue} if there exists a $v\in V \setminus \{0\}$, such that
   \begin{eqnarray}
     \label{eq:weak_eigenvalue_operator}
     A v = \lambda E v.
   \end{eqnarray}
   Then $v$ is called a \emph{(weak) eigenvector}. The \emph{(weak) resolvent}  $\rho(A)$ of $A$ is given by all values $\lambda \in \mathbb{R}$, such that $ Av - \lambda E v = f$  is uniquely solvable for all $f\in H^*$ and the inverse mapping is in $\mathcal{L}(H^*,V)$. The \emph{(weak) spectrum} is given  by $\sigma(A) = \mathbb{R} \setminus \rho(A)$. The \emph{Rayleigh quotient} is given by
\begin{eqnarray}
  \label{eq:rayleigh}
  \mu(v) = \frac{\langle A v,v\rangle}{\langle E v,v\rangle} 
  = \frac{\langle A v,v\rangle}{(v,v)}, \quad v\in V.
\end{eqnarray}
\end{definition}

We assume that the lower part of the spectrum is discrete, that is
there exist eigenvalues $0 < \lambda_1 < \ldots < \lambda_N$ of possibly 
higher multiplicity with corresponding finite dimensional eigenspace
\begin{equation*}
  \mathcal{E}_k = \mathrm{span}(u_{k,1}, \ldots u_{k,n_k}), \quad k=1,\ldots,N,
\end{equation*}
while we suppose the rest of the spectrum is bounded from below by $\Lambda > \lambda_N$ and unbounded from above.
If the latter is not the case, $A$ and $E$ are spectrally equivalent and since they induce norms on $V$ and $H$
both spaces coincide. In this paper, we will restrict ourselves to the unbounded case, only noting that 
the other case can be treated with only minor modifications.

Whenever $V$ is compactly embedded in $H$, as it is the case for eigenvalue problems on bounded domains, the spectrum consists only of eigenvalues and the previous assertions are fulfilled automatically. 

The problem we will treat in the sequel can be formulated as follows: Find the smallest eigenvalue $\lambda_1 \in \mathbb R$ 
and a corresponding eigenvector $u_1 \in V\setminus \{0\}$ such that
\begin{eqnarray*}
  A u_1 = \lambda_1 E u_1.
\end{eqnarray*}

\section{Perturbed Preconditioned inverse iteration for operators }
\label{sec:operator_pinvit}

In this section we will state an iterative method for solving operator eigenvalue
problems and show its convergence, formulated in Theorems \ref{thm:rayleigh_reduction} 
and \ref{thm:equivalenceresiduum}. For the construction and the analysis
we can rely on methods developed for generalized symmetric eigenvalue problems. 
In particular we will use an iteration based on the preconditioned steepest 
descent of the Rayleigh quotient. These methods were first analyzed in 
\cite{do80,go76}, and recent developments were achieved in \cite{bpk96,kn03, kn08}.
Generalization of such iteration schemes to operators were also considered in \cite{samokish58}.

Let us again stress that in contrast to the above references 
our iteration will be formulated in the infinite dimensional space $V$, 
not in an associated discretized space. In view of a numerical realization
of such an algorithm approximations are  unavoidable in general.
Therefore we will from the very beginning modify the Preconditioned inverse 
iteration (PINVIT) in allowing for a perturbation in each step, reflecting the finite
dimensional approximation of the involved quantities.
In Section \ref{sec:operator_ppinvit} we will discuss the errors
resulting from this approximate application of operators which 
is common in adaptive wavelet strategies \cite{cdd01}.

To state our iteration scheme, we introduce a preconditioner of the operator $A$, that 
is a symmetric operator $P: V \rightarrow V^*$, such that $A$ and $P$ are spectrally equivalent. 
Up to a scaling, this can be reformulated in the following way (cf. \cite{kamm07}):
There exists a constant $\gamma_P<1$ such that
\begin{eqnarray}
  \label{eq:precond}
  \| \mathrm{Id} - P^{-1} A \|_A \le \gamma_P,
\end{eqnarray} 
In the case of wavelets, the discretization of $P$ will be a diagonal matrix.

Now we can state the basic iteration scheme.
\begin{definition}
  \label{def:ppinvit}
  Let the starting vector $v^0 \in V$, $v^0 \not=0$, be given and define its associated Rayleigh quotient as $\mu^0 = \mu(v^0)$. A \emph{perturbed preconditioned inverse iteration (PPINVIT)} is a sequence of vectors $(v^n)_{n\ge 0}$ and associated Rayleigh quotients ($\mu^n)_{n \ge 0}$ generated by
  \begin{eqnarray*}
    \tilde v^{n+1} &=& v^n - P^{-1} ( A v^n - \mu(v^n) E v^n) + \xi^n, \\
    v^{n+1} &=& | \tilde v^{n+1} | ^{-1} \tilde v^{n+1}, \\
    \mu^{n+1} &=& \mu(v^{n+1}),
  \end{eqnarray*}
  where $(\xi^n)_{n \ge 0} \in V$ are perturbations.
\end{definition}

In order to show convergence for this sheme, we will at first generalize the results of \cite{kn08} 
to the case of this perturbed operator iteration scheme: Provided that
the perturbations are bounded by a multiple of the actual accuracy, the iteration 
generates a sequence of Rayleigh quotients converging to $\lambda_k$ such that 
the error decreases geometrically. Furthermore we will give a bound for the
rate of convergence of the associated subspaces.
In this context, it is obvious that the size of the perturbations have to match the current accuracy 
in the iterands to retain convergence.  It turns out that the eigenvalue residual
\begin{equation*}
  \rho(v) = \| A v - \lambda(v) E v \|_{A^{-1}}/\|v\|_A
\end{equation*}
is in a sense an efficient and reliable error estimator for the angle of the current iterand and the 
eigenvalue spaces. Choosing the perturbations $(\xi^{n})_{n\ge 0}$ in the order of the residuum will
guarantee the convergence of the iteration.
\begin{theorem}
  \label{thm:rayleigh_reduction}
  Let $v \in V$, $v\not=0$, such that the associated Rayleigh quotient $\lambda = \lambda(v)$ fulfills $\lambda_k \le \lambda < \lambda_{k+1}$. Furthermore assume that the perturbations $\xi$ is bounded by
  \begin{equation*}
    \| \xi \|_{A}/\|v\|_A \le \gamma_{\xi} \, \rho(v), \quad \textnormal{where} \quad \gamma = \gamma_P + \gamma_{\xi} < 1.
  \end{equation*}
  Then the next step of PPINVIT (cf. Definition \ref{def:ppinvit}) with starting vector $v$  gives $v'$ and an associated Rayleigh quotient $\lambda'=\lambda(v')$, for which 
  either $\lambda' < \lambda_k$ or $\lambda_k\le \lambda' < \lambda_{k+1}$. In the
  latter case 
  \begin{equation*}
    \frac{\lambda' - \lambda_k}{\lambda_{k+1}-\lambda'} \le q^2(\gamma,\lambda_k,\lambda_{k+1}) 
     \frac{\lambda - \lambda_k}{\lambda_{k+1}-\lambda}. 
  \end{equation*}
  Here $q$ is given by
  \begin{eqnarray*}
     q(\gamma, \lambda_k, \lambda_{k+1}) = 1 - (1-\gamma)(1 - \lambda_k / \lambda_{k+1}).
  \end{eqnarray*}
\end{theorem}
Therefore, the rate of decay is only governed by the eigenvalue gap and the quality of the preconditioner. Note that the presence of a perturbation has the same effect as  applying a preconditioner with a constant $\gamma$ instead of $\gamma_P$. 

Besides the Rayleigh quotient one is also interested in convergence to the eigenspace, which is best described by the convergence of the angle between the iterand and the eigenspace. The following theorem states that the angle for the smallest eigenvalue is  controlled by the magnitude of the eigenvector residual.
\begin{theorem}
  \label{thm:equivalenceresiduum}
  Let $v \in V$, $v\not=0$, such that for the associated Rayleigh quotient 
  $\lambda = \lambda(v)$ fulfills $\lambda_1 \le \lambda < \lambda_{2}$. Denote the angle 
  in terms of the scalar product $\langle A \cdot, \cdot \rangle$ between $v$ and the 
  eigenspace by $\phi_A(v,\mathcal E_1)$. Then 
  \begin{equation*}
    \sin \phi_A(v, \mathcal E_1) \le 
    \sqrt{\frac{\lambda_2}{\lambda_1} \cdot \frac{\lambda(x) - \lambda_1 }{\lambda_2 - \lambda(x)}}.
  \end{equation*}
  Moreover the eigenvector residual controls the angle, i.e.
  \begin{equation*}
    \frac{\lambda_1}{3\lambda(v)}\rho(v)  \le \sin \phi_A(v, \mathcal L_1) \le 
        \frac{\lambda_2}{\lambda_2 -\lambda(v)} \rho(v).
  \end{equation*}
\end{theorem}
Therefore convergence of the Rayleigh quotient towards $\lambda_1$ assures convergence 
of the angle between the iterands and the corresponding eigenspace $\mathcal E_1$. 
Concerning the perturbation, its magnitude may be chosen proportional to the current 
error in the subspaces. 

The above statements will be proven by reducing the problem to a more simpler model
case.

\subsection{A model case analysis}\label{sec:model}

The aim of this section is to prove the convergence rate of a preconditioned
inverse iteration for a special eigenvalue problem given by a bounded
operator on a Hilbert space. Later on the setting of the previous part
can be transformed to fulfill the specialized assumptions.

Suppose that a Hilbert space $X$ with scalar product $(\cdot,\cdot)_X$
is given. Furthermore let $B:X\rightarrow X$ be a bounded operator
such that 
\begin{equation*}
  \inf_{x \in X \setminus \{0\}} \frac{(Bx,x)_X}{(x,x)_X} = 0. 
\end{equation*}
where the top of the spectrum consists of discrete eigenvalues 
$\mu_N < \ldots < \mu_1$ with corresponding finite dimensional eigenspace $\mathcal X_k$.  
We are looking for the biggest eigenvalues $\mu_i$ of the eigenvalue
problem
\begin{equation*}
  B x = \mu x.
\end{equation*}

Furthermore suppose we are given an preconditioner $T: X \rightarrow X$
such that 
\begin{equation*}
  \| \mathrm{Id} - T \|_X \le \gamma_{T}.
\end{equation*}
Then we define the following iteration:
\begin{equation}\label{modelit}
  x' = x + \frac{1}{\mu(x)} T( B x - \mu(x) x) + \eta, \quad
  \mu(x) = \frac{(Bx,x)_X}{(x,x)_X}, 
\end{equation}
where $\eta \in X$ is again a perturbation. The eigenvalue residual 
is defined by
\begin{equation*}
  \rho(x) = \| \frac{1}{\mu(x)}( B x - \mu(x) x) \|_X/\|x\|_X.
\end{equation*}

Regarding the convergence of the sequence generated by \ref{modelit} we can state the following
estimate.
\begin{theorem} \label{PPINVITconv}
  Let $x \in X$, $x\not=0$, such that for the associated Rayleigh quotient $\mu = \mu(x)$ 
  fulfills $\mu_{k+1} < \mu \le \mu_{k}$. Furthermore assume that the perturbation $\eta$ 
  is bounded by 
  \begin{equation*}
    \| \eta \|_X/\|x\|_X \le \gamma_{\eta} \, \rho(x), \quad \textnormal{where} \quad
    \gamma = \gamma_P + \gamma_{\eta} < 1.
  \end{equation*}
  The above iteration step applied to a vector $x$ then gives an output $x'$ and an associated 
  Rayleigh quotient $\mu'=\mu(x')$, for which either $\mu' > \mu_k$ or 
  $\mu_{k+1} \le \mu' < \mu_{k}$. In the latter case, 
 \begin{equation*}
    \frac{\mu_k - \mu'}{\mu' - \mu_{k+1}} \le \sigma^2 \frac{\mu_k - \mu}{\mu - \mu_{k+1}},
    \quad \sigma=1-(1-\gamma)\frac{\mu_k - \mu_{k+1}}{\mu_k}.
  \end{equation*}
  \label{thm:pinvit_inverse}
\end{theorem}
\begin{proof}
Neglecting the perturbation the proof of this theorem can be taken almost verbatim from the proof of
Theorem 1.1 of \cite{kn08} for the matrix case, where $\mu_{\min}=0$. The notation has been chosen
identical to the one from  \cite{kn08} in order to simplify this transition. In the presence of 
perturbations, we see that the scaled iterand $\mu(x) x'$ fulfills
\begin{equation*}
  \mu(x) x' = Bx - (I-T)(Bx - \mu(x) x) + \mu(x) \eta.
\end{equation*}
Hence $\mu(x) x'$ lies in a ball with radius $\gamma \| B x - \mu(x) x\|_X$ and center $Bx$ since
the distance between $\mu(x) x'$ and $B x$ can be estimated by
\begin{align*}
  \| (I-T) (B x - \mu(x) x)  + \eta \|_X & \le \gamma_T \| B x - \mu(x) x \|_X  
    +  \gamma_{\eta} \|B x - \mu(x) x \|_X \\ 
  & = \gamma \| B x - \mu(x) x \|_X.
\end{align*}
From then on the proof proceeds as before.   
\end{proof}

The error estimation of Theorem \ref{thm:equivalenceresiduum} also has a 
counterpart in this setting.
\begin{theorem}
  Let $x \in X$, $x\not=0$, such that for the associated Rayleigh quotient $\mu = \mu(x)$ fulfills
  $\mu_2 < \mu \le \mu_1$. Denote the angle in terms of the inner product on $X$ between $x$ and 
  the eigenspace $\mathcal X_1$ by $\phi_X(x,\mathcal X_1)$. Then 
  \begin{equation*}
    \sin \phi_X(x, \mathcal X_1) \le \sqrt{\frac{\mu_1 - \mu(x)}{\mu(x) - \mu_2}}.
  \end{equation*}
  Moreover the eigenvector residual controls the angle, i.e.
  \begin{equation*}
    \frac{\mu(x)}{3 \mu_1} \rho(x) \le \sin \phi_X(x, \mathcal X_1) \le 
    \frac{\mu(x)}{\mu(x)-\mu_2} \rho(x).
  \end{equation*}
\end{theorem}
\begin{proof}
  Decompose $x$ orthogonally into $x = x_{\|} + x_{\perp}$, where $x_{\|}\in \mathcal X_1$ and
  $x_{\perp}\in \mathcal X^{\perp}_1$. By the definition of the Rayleigh quotient
  and by orthogonality
  \begin{align*}
    \mu(x) ( \|x_{\|}\|_X^2 + \|x_{\perp}\|_X^2 ) & = \mu(x) \|x\|_X^2= \|x\|_B^2 
    = \|x_{\|}\|_B^2 + \|x_{\perp}\|_B^2 \\
    & \le \mu_1 \|x_{\|}\|_X^2 + \mu_2 \|x_{\perp}\|_X^2.
  \end{align*}
  Then one can estimate the angle by
  \begin{equation*}
    \sin \phi_X = \frac{\|x_{\perp}\|_X}{\|x\|_X}\le \frac{\|x_{\perp}\|_X}{\|x_{\|}\|_X}
    \le \sqrt{\frac{\mu_1 - \mu(x)}{\mu(x) - \mu_2}}.
  \end{equation*}

  For the upper bound of the second inequality we use the previous estimate and combine it
  with the Temple-Kato inequality
  \begin{equation*}
    (\mu_1 - \mu(x))(\mu(x)-\mu_2) \le \|B x - \mu(x) x\|_X^2 / \|x\|_X^2,
  \end{equation*}
  which will directly give the desired result.
  In order to prove the Temple-Kato inequality let $\bar{\mu} = \mu(x)$ for convenience. Then by the
  spectral calculus it follows that
  \begin{equation*}
   \| B x - \bar \mu x \|_X^2 =  
    \int_{\sigma(B)} (\mu- \bar \mu)^2 \mathrm d(E_{\mu} x,x).
  \end{equation*}
  Then the inequality follows directly from
  \begin{equation*}
    (\mu - \mu_{k+1}) (\mu - \mu_k) \ge 0, \quad \mu \in \sigma(B).
  \end{equation*}


  For the lower estimate we see that
  \begin{align*}
    \rho(x) \le \frac{1}{\mu} \| B x - \mu_1 x\|_X / \|x\|_X   + \frac{\mu_1 - \mu(x)}{\mu(x)}
  \end{align*}
  We estimate the terms separately:
  \begin{align*}
    \| B x - \mu_1 x \|_X^2 &= \int_{\sigma(B)} (\mu - \mu_1)^2 \; \mathrm d(E_{\mu}x,x)_X \\
    &=\int_{\sigma(B)\setminus \{\mu_1\}} (\mu - \mu_1)^2 \; \mathrm d(E_{\mu}x,x)_X 
    \le \mu_1^2 \|x^{\perp}\|_X^2.
  \end{align*}
  The other term gives for $x_1 = x_{\|}/\|x_{\|}\|_X$ and $\|x\|_X=1$
  \begin{align*}
    \mu(x) - \mu_1 & = (B x,x) - (B x_1, x_1) \\
    & = ( B(x-x_1),(x-x_1) ) - \mu(x) ( x-x_1, x-x_1) \\
    & \le \mu_1 \| x-x_1 \|^2.
  \end{align*}
  It then follows by geometric reasoning (cf. Lemma 2 in \cite{drsz08})  that 
  \begin{equation*}
    \mu_1 - \mu(x)  \le 2 \mu_1 \| x^{\perp}\|_X^2/\|x\|_X^2.
  \end{equation*}
  Putting everything together gives
  \begin{equation*}
    \rho(x) \le \frac{\mu_1}{\mu(x)} \sin\phi_X + 
    \frac{2 \mu_1 \sin \phi_X}{\mu(x)} \sin \phi_X \le \frac{3 \mu_1}{\mu(x)} \sin \phi_X, 
  \end{equation*}
  and hence the assertion.
\end{proof}

\subsection{Reduction to the model case}

The aim of this part is to show how to transform the general setting of Section 
\ref{sec:operator_formulation} to the case analyzed in the previous section \ref{sec:model}. 

Based on the transformation $A u = \lambda E u$ to $A^{-1}E u = \lambda^{-1} u$
we set
\begin{equation*}
  B = A^{-1}E: V \rightarrow V, \quad \mu=\lambda^{-1}.
\end{equation*}
On the space $V$ we introduce the inner product $(v,w)_V = \langle A v, w \rangle$
and the induced norm $\|v\|_V = \sqrt{(v,v)_V} = \|v\|_A$. 
Then the upper part of the spectrum of $B$ consists of discrete 
eigenvalues $\mu_k = \lambda_k^{-1}$ with finite dimensional eigenvalues.
Since the spectrum of $A$ is by assumption unbounded it follows that
the infimum of the spectrum is zero. The Rayleigh quotient is then
defined as 
\begin{equation*}
  \label{eq:rayleigh_inverse}
  \mu(v) = \frac{(B v,v)_V}{(v,v)_V} =
  \frac{ \langle E v,v\rangle}{\langle A v,v\rangle}
\end{equation*}
As the preconditioner we set
\begin{equation*}
  T = P^{-1}A: V \rightarrow V.
\end{equation*}
For the quality of the preconditioner it follows due to symmetry with
respect to $(\cdot, \cdot)_V$ that
\begin{equation*}
  \| \mathrm{Id} - T\|_V = \sup_{v \in V \setminus \{0\}} ( (\mathrm{Id} - P^{-1} A) v,v)_V = 
  \|\mathrm{Id} - P^{-1} A\|_A \le \gamma_P.
\end{equation*}

Our aim is now to show that the iterands and the Rayleigh quotients
of the iterations defined by PPINVIT (Definition \ref{def:ppinvit}) and 
by the model iteration defined by Equation \eqref{modelit} coincide, 
provided that $\eta = \xi$.
For the Rayleigh quotient, it follows
directly from equation \eqref{eq:rayleigh_inverse} that $\mu(v) = \lambda(v)^{-1}$.
For the next iterand 
\begin{align*}
  v' &= v + \frac{1}{\mu(v)} T (B v - \mu(v) v) + \eta \\
  &= v + \lambda(v) P^{-1} A ( A^{-1}E v - \lambda(v)^{-1} v) + \eta\\
  &= v - P^{-1}( A v - \lambda(v) E v) + \eta,
\end{align*}
which is the iteration of Theorem \ref{thm:rayleigh_reduction}
if we set $\eta = \xi$.

For the estimation of the convergence of the Rayleigh quotients, 
the factor can be transformed to
\begin{equation*}
  \frac{\lambda_{k+1}-\lambda_k}{\lambda_{k+1}} = 
  \frac{1/\mu_{k+1} - 1/\mu_k}{1/\mu_{k+1}} = \frac{\mu_k - \mu_{k+1}}{\mu_k}.
\end{equation*}
For the fractions of the Rayleigh quotients
\begin{equation*}
  \frac{\lambda(v) - \lambda_k}{\lambda_{k+1}-\lambda(v)} = 
  \frac{ 1/\mu - 1/\mu_k}{1/\mu_{k+1} - 1/\mu} = 
  \frac{\mu_k - \mu}{\mu-\mu_{k+1}} \cdot \frac{\mu_{k+1}}{\mu_k}.
\end{equation*}
and analogously also for $\lambda'$. As the factor $\mu_{k+1}/\mu_k$ 
appears on both sides the inequality of Theorem \ref{thm:rayleigh_reduction}
follows from Theorem \ref{thm:pinvit_inverse}.

By direct calculation it follows that
\begin{equation*}
  \| \frac{1}{\mu(v)} (B v - \mu(v) v) \|_V = \|A v -\lambda(v) E v \|_{A^{-1}}.
\end{equation*}
and hence the perturbations are bounded by the same term. Moreover the
upper and lower bounds of Theorem \ref{thm:equivalenceresiduum} follow
immediately.

Note that this sort of transformation has also been used in \cite{kn03} to get
rid of the mass matrix.

\section{Inexact operator applications}
\label{sec:operator_ppinvit}

Considering a numerical realization of the PPINVIT, the interpretation of the perturbation $\xi$ is rather natural: In the calculation of the residual $r(v) = Av - \lambda(v)Ev$, the corresponding coefficient vector generally has infinitely many entries, 
so that the perturbation will be related to the error of the finite dimensional approximations of the actual residuals.
In this section, we will show that the convergence of PPINVIT is retained as long as the approximate applications of the operators $A$ and $E$ are kept proportional to the current subspace error, which may also be measured by the residual through Theorem \ref{thm:equivalenceresiduum}. We will start in Section \ref{sec:inexapplop} by showing that inexact operator applications proportional to some $\varepsilon > 0$ result in an approximate residual $r_{\varepsilon}(v)$ which approximates  $r(v)$ up to a constant times $\varepsilon$. The main difficulty here stems from the nonlinearity of the Rayleigh quotient and from the fact that in contrast to the finite dimensional case, we will have to deal with different norms for $H$ and $V$. In the next section, we then devise an algorithm which determines an appropriate accuracy for the residual to preserve convergence, while the operator applications involved only have to be carried out proportional to the current subspace error.

To start with, let us introduce the approximate operators and a measure for their quality.

\begin{definition}
  \label{def:approx_op}
  For all $v \in V$ and $\varepsilon >0$ let $A_{\varepsilon}(v)$ and $E_{\varepsilon}(v)$ be approximation of $A v$ and $E v$ respectively, such that
  \begin{eqnarray*}
    \|A_{\varepsilon} (v) - A v \|_{*} \le \varepsilon \|v\|, \quad
    |E_{\varepsilon} (v) - E v |_{*} \le \varepsilon | v |.
  \end{eqnarray*}
  Furthermore define the \emph{perturbed Rayleigh quotient} as
  \begin{eqnarray*}
    \mu_{\varepsilon}(v) = \frac{\langle A_{\varepsilon}(v), v \rangle}{\langle E_{\varepsilon}(v), v \rangle}.
  \end{eqnarray*}
\end{definition}
In the following we will assume that  $P^{-1}$ can be applied exactly, which causes no further restrictions: If only an approximate operator $\tilde P$ of $P$ is available which is still spectrally equivalent to $A$, one simply uses $\tilde P$ instead of $P$ in all calculations. In the analysis, the constant $\gamma_P$ of equation (\ref{eq:precond}) is to be replaced by the corresponding perturbed one. In addition, the iteration may also be generalized by using a different preconditioner $P_n$ in  each step. The convergence results can be extended to such cases if the family of preconditioners $(P_n)_{n\ge 0}$ is uniformly spectrally equivalent to $A$, i.e. 
\begin{eqnarray*}
  \| I - P_{n}^{-1} A \|_{A} \le \gamma_P
  \quad \textnormal{for all}  \quad n \ge 0.
\end{eqnarray*}
In the prototype example of wavelets considered in Section \ref{sec:wavelet} 
the discretization of $P$ will lead to a diagonal matrix which can be applied exactly.

With Definition \ref{def:approx_op}, the following recursion allows a finite dimensional implementation of the PPINVIT algorithm.
\begin{definition}
  \label{def:ppinvit2}
  Let a vector $v$, $v\not=0$, and a tolerance $\varepsilon>0$ be given. Define the approximate Rayleigh quotient  $\mu_{\varepsilon} = \mu_{\varepsilon}(v)$ associated to $v$. We let  
  \begin{eqnarray*}
    \tilde v' &=& v'_{\varepsilon} - P^{-1} ( A_{\varepsilon}(v_{\varepsilon}) - \mu_{\varepsilon} E_{\varepsilon}( v_{\varepsilon}) ), \\
    v'_{\varepsilon} &=& | \tilde v'_{\varepsilon} | ^{-1} \tilde v'_{\varepsilon}, \\
    \mu'_{\varepsilon} &=& \mu_{\varepsilon}(v'_{\varepsilon}).
  \end{eqnarray*} 
\end{definition}
Comparing this recursion with the original Definition \ref{def:approx_op} gives $\xi = B^{-1} (r_{\varepsilon}(v)-r(v))$, so we will have to bound the $A$-norm of this expression in terms of the residual $\rho(v)$, cf. Theorem \ref{thm:rayleigh_reduction}.

%

In the course of our analysis, it will be important to keep track of the induced error with respect to the given tolerance $\varepsilon$. This is done via constants $c_0,c_1,c_2,c_3$ that will be specified  in the proofs. As we are only interested in a qualitative statement, it suffices to know that these constants can be bounded independent of $\varepsilon$ and the current vector $v$, as long as certain requirements are fulfilled. General bounds for $c_0,c_1,c_2,c_3$ will involve constants such as $\alpha, \lambda_1, \lambda_2, \delta_0, \sigma_0, \sigma_1$ etc.   

\subsection{Inexact application of operators}\label{sec:inexapplop}

First we will investigate the difference between the approximate and the exact Rayleigh quotient.
\begin{lemma}
  \label{lemma:approx_rayleigh}
  Let $c_0 = \min\{1/2, \sigma_0\}$, $0<\varepsilon \le c_0$, $v \in V$, $v \not=0$, and $\lambda_1 \le \mu(v) < \lambda_2$. Then the approximate Rayleigh-quotient is bounded by
  \begin{eqnarray*}
    | \mu_{\varepsilon}(v) -  \mu(v) | \le  c_1 \varepsilon,
  \end{eqnarray*}
  where $c_1$ can bounded independently of $v$ and $\varepsilon$.
\end{lemma}
\begin{proof}
  Writing out the approximate Rayleigh quotient gives
  \begin{eqnarray}
    \label{eq:approx_rayleigh_estimate}
    \mu_{\varepsilon}(v)   
    = \frac{\langle A v, v \rangle + \langle A_{\varepsilon}(v)-A v, v \rangle}
    {\langle E v, v \rangle + \langle E_{\varepsilon} (v) -E v, v \rangle}.
  \end{eqnarray}
  From the definition of the approximate operators $A_{\varepsilon}$ and $E_{\varepsilon}$ it follows that
  \begin{eqnarray*}
    | \langle A_{\varepsilon}(v)-A v , v \rangle | \le \varepsilon \|v\|^2
    \le \frac{\varepsilon}{\sigma_0} \|v\|_A^2 , \quad
    | \langle E_{\varepsilon}(v)-E v , v \rangle | \le \varepsilon |v|^2,
  \end{eqnarray*}
  where we used the norm equivalence of $A$ and $\|\cdot\|$, equation (\ref{eq:A_constants}). Inserting these estimates in equation (\ref{eq:approx_rayleigh_estimate}) gives
  \begin{eqnarray*}
    \frac{1-\frac{\varepsilon}{\sigma_0}}{1+\varepsilon} \mu(v) \le 
    \mu_{\varepsilon}(v) \le \frac{1+\frac{\varepsilon}{\sigma_0}}{1-\varepsilon} \mu(v),
  \end{eqnarray*}
  by noting that $\varepsilon \le c_0$. To estimate the difference in the Rayleigh quotients, we subtract $\mu(v)$ to obtain
  \begin{eqnarray*}
    - \frac{1 + \sigma_0^{-1}}{1 + \varepsilon} \mu(v) \varepsilon \le
    \mu_{\varepsilon}(v) - \mu(v) \le
    \frac{1+\sigma_0^{-1}}{1-\varepsilon} \mu(v) \varepsilon.
  \end{eqnarray*}
  Now setting
  \begin{eqnarray*}
    c_1 = \max\left\{\frac{1 + \sigma_0^{-1}}{1 + \varepsilon} \mu(v), 
    \frac{1+\sigma_0^{-1}}{1-\varepsilon} \mu(v)\right\}  = \frac{1+\sigma_0^{-1}}{1-\varepsilon} \mu(v)
  \end{eqnarray*}
  gives $|\mu_{\varepsilon}(v) - \mu(v) | \le c_1 \varepsilon$. Furthermore $c_1$ can be bounded from above by
  \begin{eqnarray*}
    c_1 \le 2(1+\sigma_0^{-1}) \lambda_2
  \end{eqnarray*}
  since $\mu(v) \le \lambda_2$ and $\varepsilon \le 1/2$.
\end{proof}

Next we will estimate the difference between the approximated and exact residual with respect to a norm which will come in handy later on.
\begin{lemma}
  \label{lemma:residual_approx}
  Let $0<\varepsilon \le c_0$, $v \in V$, $v \not=0$ and $\lambda_1 \le \mu(v) < \lambda_2$. 
Then there exists a constant $c_2$ such that for 
  \begin{eqnarray*}
    r(v) = A v - \mu(v) E v, \quad 
    r_{\varepsilon}(v) := A_{\varepsilon}(v) - \mu_{\varepsilon}(v) E_{\varepsilon}(v),
  \end{eqnarray*}
 the difference of the exact and the approximate residual can be bounded by
  \begin{eqnarray*}
    \|  r_{\varepsilon}(v) - r(v) \|_{A^{-1}}/ \|v\|_A \le c_2 \varepsilon.
  \end{eqnarray*}
  Furthermore $c_2$ can be bounded independently of $v$ and $\varepsilon$.
\end{lemma}
\begin{proof}
  The norm of the difference in the residual can be estimated by
  \begin{eqnarray*}
    \| r_{\varepsilon}(v) - r(v) \|_{*} 
    &\le& \| A_{\varepsilon}(v) - A v \|_{*}
    + \mu(v) \| E_{\varepsilon}(v) - E v \|_{*} \\
    && + |\mu_{\varepsilon}(v) - \mu(v) | \|E v\|_{*}
    + |\mu_{\varepsilon}(v) - \mu(v) |  \| E_{\varepsilon}(v) - E v \|_{*},
  \end{eqnarray*}
  where we used the triangle inequality. We have to estimate $\|Ev\|_{*}$ and 
  $\| E_{\varepsilon}(v) - E v \|_{*}$ in the norm stemming from $H^*$. For that purpose, 
  for $f \in H^*$, the dual norm can be estimated by $\| \left. f \right|_V \|_{*} \le \alpha | f  |_{*}$. 
  Using this, the definition of the approximate operators (definition \ref{def:approx_op}) 
  and the result on the approximate Rayleigh quotients (lemma \ref{lemma:approx_rayleigh}) gives
  \begin{eqnarray*}
    \| r_{\varepsilon}(v) - r(v) \|_{*} 
    &\le& \varepsilon \|v\|
    + \mu(v) \alpha^2 \varepsilon \|v\|  
    + c_1 \varepsilon \alpha^2 \|v\|
    + c_1 \varepsilon^2 \alpha^2 \|v\| \\
    &=& \{1 + \alpha^2 [\mu(v)  + c_1(1+\varepsilon)]\} \|v\| \varepsilon.
  \end{eqnarray*}
  Since $A^{-1}$ is bounded as a mapping between $V^{*}$ and $V$ with norm $\sigma_0^{-1/2}$, setting
  \begin{eqnarray*}
    c_2 = \sigma_o^{-1/2} \{1 + \alpha^2 [\mu(v)  + c_1(1+\varepsilon)]\},
  \end{eqnarray*}
  we can estimate
  \begin{eqnarray*}
    \|r_{\varepsilon}(v) - r(v)\|_{A^{-1}}  \le c_2 \varepsilon \|v\|.
  \end{eqnarray*}
  Again estimating $\mu(v) \le \lambda_2$, $\varepsilon \le 1/2$ and $c_1$ by the upper bound in lemma \ref{lemma:approx_rayleigh}, $c_2$ can be bound from above independent of $\varepsilon$ and $v$.
\end{proof}

\subsection{Error estimation from approximate residuals}

In view of the convergence results for the perturbed PINVIT as stated in Theorem \ref{thm:rayleigh_reduction}, 
we may still guarantee convergence of the perturbed algorithm if we admit for perturbations for which  $\gamma = \gamma_P + \gamma_{\xi}<1$, where $\gamma_P<1$ is the constant entering via the preconditioner. For simplicity, we fix $\gamma_{\xi}:= \frac {1-\gamma_P}{2}$ in the sequel. To retain convergence, we have to bound the perturbation by the \emph{accuracy criterion}
\begin{equation}
  \label{eq:accuracy}
  \|P^{-1}(r_{\varepsilon}(v)-r(v))\|_A / \|v\|_A \le \gamma_{\xi} \rho(v) ;
\end{equation}
using Lemma \ref{lemma:residual_approx}, it is obvious that this  can be guaranteed
if only $\varepsilon$ is chosen small enough. However, approximating the residual with 
more accuracy than necessary at the present stage may lead to unnecessary costs. 
Therefore, this section is dedicated to the analysis of an algorithm iteratively
determining $\varepsilon(v)$ (for a given iterand $v$) in a way that for each step,  
$\varepsilon(v) \gtrsim \max( \tau, \rho(v) ),$ where $\tau$ is a target accuracy for the residual. Note that by this,
a target accuracy for the subspace error may be fixed, cf. Theorem \ref{thm:equivalenceresiduum}.

As a first step we define an efficient and reliable error estimator for $\rho(v)$ provided
that the tolerance $\varepsilon$ used for the computation is small
enough.
\begin{lemma}
  Let the estimator of the residual $\rho(v)$ be defined by
  \begin{equation*}
    \rho_{\varepsilon}(v) = \|r_{\varepsilon}(v)\|_{P^{-1}}/\|v\|_{P}.
  \end{equation*}
  For sufficiently small $\varepsilon$ it is efficient as well as 
  reliable in the sense that
  \begin{equation*}
    (1+\gamma_P)^{-1} \rho_{\varepsilon}(v) - c_2 \varepsilon \le \rho(v) \le 
    (1-\gamma_P)^{-1} \rho_{\varepsilon}(v) + c_2 \varepsilon.
  \end{equation*}
  Furthermore one may choose $\varepsilon \gtrsim \rho(v)$ such that
  $\rho_{\varepsilon}(v) \sim \rho(v)$.
  \label{lemma:residual_estimator}
\end{lemma}
\begin{proof}
  The given inequalities are a direct result of the norm equivalence between $A$ and $P$ 
  as well as Lemma \ref{lemma:residual_approx}.    

  Now choosing $\varepsilon \le (2 c_2)^{-1} \rho(v)$ results in 
  \begin{equation*}
    \frac{2}{3}(1+\gamma_P)^{-1} \rho_{\varepsilon}(v) \le \rho(v) \le 
    2 (1-\gamma_P)^{-1} \rho_{\varepsilon}(v),
  \end{equation*}
  which shows the second assertion.
\end{proof}
Now we can use this estimator to
test if the accuracy criterion \eqref{eq:accuracy} is already met.
\begin{lemma}
  The preconditioned approximate residual $P^{-1} r_{\varepsilon}(v)$ fulfills
  the accuracy criterion of Equation \eqref{eq:accuracy} if 
  \begin{equation}
    \varepsilon \le c_3 \rho_{\varepsilon}(v),
    \label{eq:accuracy_test}
  \end{equation}
  where $c_3$ is a constant independent of $v$ and $\varepsilon$. 
  Furthermore $\varepsilon$ can be chosen such that $\varepsilon \gtrsim \rho(v)$.
  \label{lemma:accuracy_test}
\end{lemma}
\begin{proof}
  The idea is to bound both sided of Equation \eqref{eq:accuracy} and to 
  require that the bounds satisfy the inequality. Applying Lemma
  \ref{lemma:residual_approx} gives for the right hand side
  \begin{equation*}
    \|P^{-1}(r_{\varepsilon}(v)-r(v)\|_{A} \le \|P^{-1}A\|_A \|A^{-1}(r_{\varepsilon}(v)-r(v))\|_A
    \le (1+\gamma_P)^{1/2} c_2 \varepsilon,
  \end{equation*}
  while for the left hand side we simply apply the lower bound of 
  Lemma \ref{lemma:residual_estimator}. If we choose
  \begin{equation*}
    c_3 = [(1+\gamma_P)^{1/2} c_2 + \gamma_{\xi} c_2]^{-1} \frac{\gamma_{\xi}}{1 - \gamma_P}
  \end{equation*}
  the assertion follows readily.

  Next we have to prove that solutions of this inequality do not get
  too small. In view of Lemma \ref{lemma:residual_approx} choosing 
  $\varepsilon \le (2c_2)^{-1} \rho(v)$ will lead to
  \begin{equation*}
    \rho_{\varepsilon} \ge  \frac{1-\gamma_P}{2} \rho(v).
  \end{equation*}
  Therefore for all 
  \begin{equation*}
    \varepsilon \le \min( (2c_2)^{-1}, c_3 (1-\gamma_P) /2)\rho(v)  
  \end{equation*}
  the inequality holds, showing that one can choose $\varepsilon \gtrsim \rho(v)$.
\end{proof} 

\subsection{Convergence of PPINVIT}

Based on the estimators of the previous subsection a simple algorithm can be 
devised to calculate a suitable tolerance $\varepsilon$: Starting with an 
initial guess we successively halve $\varepsilon$ until the accuracy criterion of Equation \eqref{eq:accuracy_test}
is met. For this $\varepsilon$ we can still guarantee that it is proportional
to the actual error indicated by $\rho(v)$.

The last ingredient for our PPINVIT algorithm is a stopping criterion indicating when the exact residual has dropped below the given target accuracy $\tau$. Note that in this context the work load in might still become too high if the algorithm tries to determine $r_{\varepsilon}$ according to Equation \eqref{eq:accuracy_test}, while $r$ is already smaller that $\tau$. To prevent this situation, we add a stopping test based on Lemma \ref{lemma:residual_estimator}.
Combining this with the error estimator and one step of PPINVIT leads
to the following algorithm:\\

\begin{algorithmic}
  \STATE $\mathrm{PPINVIT\_STEP}(v,\tau) \rightarrow v'$ \\
\hrulefill\\

  \STATE $\varepsilon := c_0$ 
  \LOOP
    \STATE $r := A_{\varepsilon}(v) - \mu_{\varepsilon}(v) E_{\varepsilon}(v)$
    \STATE $\rho := \|r\|_{P^{-1}} / \|v\|_P$
    \IF{$(1-\gamma_P)^{-1}\rho + c_2 \varepsilon \le \tau$}
      \STATE \emph{\{target accuracy reached, criterion Lemma \ref{lemma:residual_estimator}\}}
      \STATE \textbf{return} $v$ 
    \ENDIF
    \IF{$\varepsilon \le c_3 \rho$}  
      \STATE \emph{\{accuracy for residual reached, criterion Lemma \ref{lemma:accuracy_test}\}}
      \STATE \textbf{return} $v' := v - P^{-1} r$
    \ENDIF
    \STATE $\varepsilon := \varepsilon/2$
  \ENDLOOP\\
\hrulefill\\
\end{algorithmic}

A similar algorithm has already been used for the determination of refined 
sets in the context of adaptive treatment of PDEs, cf. the GROW procedure in \cite{ghs07}.
Note also that the intermediate computations of $ r_{\varepsilon}(v) $ need not cause 
much additional computational effort; indeed, when using the APPLY algorithm from \cite{drsz08},
the computation of the residual with a lower accuracy is part of 
the computation of the residual with a higher accuracy, therefore, intermediate values 
can be used to estimate the size of $r(v)$ and to effectively approximate a suitable 
error tolerance $\varepsilon$.

The properties of the algorithm $\mathrm{PPINVIT\_STEP}$, which follow from combining the general theory of convergence,
Theorem \ref{thm:rayleigh_reduction}, with the results of this section are compiled in 

%

\begin{theorem}
  For all starting vectors $v\not=0$ such that the Rayleigh quotient fulfills 
  $\lambda(v) < \lambda_2$, iterating PPINVIT\_STEP generates a sequence of vectors, for which the error
in the Rayleigh quotients decreases geometrically according to Theorem \ref{thm:rayleigh_reduction} 
where, with the above choice of $\gamma_{\xi}$, $\gamma=(1+\gamma_P)/2$. In each step, the accuracy $\varepsilon$ is proportional to $\rho(v)$. \\
Combining Theorems \ref{thm:rayleigh_reduction} and \ref{thm:equivalenceresiduum}
shows that along with the Rayleigh quotients, the residual $\|Av - \lambda(v) Ev\|_{A^{-1}}$
will decrease. Hence the algorithm terminates after finitely many steps. \\
For the final iterand $v^{o}$, there holds $\rho(v^{o}) \leq \tau$, so that in turn, $\sin \phi_A(v^{o}, \mathcal E_1)  \lesssim \tau$. The maximal accuracy $\varepsilon$ needed for the approximate applications of
  $A$ and $E$ stays bounded by
  \begin{equation*}
    \varepsilon \gtrsim \max( \tau, \rho(v) ).
  \end{equation*}
\end{theorem}

Combining the algorithm with a time to
time coarsening worked out in more detail in \cite{drsz08} leads to an optimally convergent
algorithm. It is likely that as in the case of boundary value problems
\cite{ghs07} a coarsening of the updates $P^{-1}r_{\varepsilon}(v)$
instead of the iterands gives a convergent algorithm with improved
performance however we were not able to prove this conjecture yet. 

\section{Application to planar eigenvalue problems}
\label{sec:wavelet}

In this section the application of the abstract eigenvalue solver is discussed
for the model case of a planar eigenvalue problem discretized by a stable
wavelet base. The corresponding adaptive procedure is superior to uniform
refinement if the corresponding eigenfunction has higher
regularity in terms of certain Besov spaces than in the scale of Sobolev
spaces.

Hence first we will address the Besov regularity of the eigenfunctions
and show that the eigenfunction can be approximated with an arbitrary high
rate, provided that the wavelet ansatz function have sufficiently many vanishing
moments.

After that we provide some numerical tests for the case of an L-shaped domain 
with piecewise linear wavelets. The numerical results approve the theoretical 
predictions.

As a model we introduce the Poisson eigenvalue problem on a polygonal
domain with Dirichlet boundary condition: Let $\Omega \subset \mathbb R^2$ be a bounded open polygonal domain with
vertices $x^{(1)}, \ldots, x^{(d)}$ such that the interior angles $\alpha_i$ at $x^{(i)}$
satisfy $0 < \alpha_i < 2 \pi$ for all $i=1,\ldots,d$. Let the
Poisson eigenvalue problem with homogenous Dirichlet boundary condition 
\begin{align}  \label{eq:poisson}
  -\Delta u &= \lambda u, \quad \textnormal{on } \Omega, \\
  u &= 0, \quad \textnormal{on } \partial \Omega
\end{align}
be given.

\subsection{Regularity}

In what follows we want to determine the regularity of the eigenfunctions $u$
in terms of Besov norms. To achieve this we employ the regularity results
with respect to certain weighted Sobolev spaces as were described in \cite{kmr97},
see also the references therein.

For the construction of the weighted Sobolev spaces define for each
vertex $x^{(i)}$ an infinitely differentiable cut-off function $\xi_i$ that is equal to one in 
a sufficiently small neighborhood of $x^{(i)}$ and zero outside, such that
the support of the functions $\xi_i$ do not intersect. Define 
$\xi_0 := 1 - \sum_{i=1}^d \xi_i$.

\begin{definition}
  Let the domain $\Omega$ be given as above. For the integer $l\ge 0$ 
  and $\beta \in \mathbb R$ define the weighted Sobolev space $V^{l}_{2,\beta}$
  as the closure of $C^{\infty}(\Omega)$ with respect to the norm
  \begin{equation*}
    \|u\|_{V^{l}_{2,\beta}(\Omega)}^2 = \| \xi_0 u\|_{W^{l}_2(\Omega)}^2 
    + \sum_{i=1}^{d} \sum_{|\alpha|\le l} \| \rho_i^{|\alpha|+\beta-l} \partial^{\alpha}(\xi_i u)\|_{L_2(\Omega)}^2,
  \end{equation*}
  where $\rho_i$ is the Euclidean distance to $x^{(i)}$ and $W^l_2(\Omega)$ are
  classical Sobolev spaces. Moreover define
  the weighted Sobolev space $\hat V^{l}_{2,\beta}(\Omega)$ as the closure
  of $C^{\infty}_0(\Omega)$ with respect to $\|u\|_{V^{l}_{2,\beta}(\Omega)}$.
  \label{def:weightedsobolev}
\end{definition}
Note that this definition is a slightly simplified version of the definition 
in paragraph 6.2.1 of \cite{kmr97}. Here we restrict ourselves to the case
of polygonal domains and scalar $\beta$.

From Theorem 6.6.1 in \cite{kmr97} we can deduce the following regularity result:
\begin{theorem}
  The operator $-\Delta$  is an isomorphism between $\hat V^l_{2,l-1}(\Omega)$ 
  and $V^{l-2}_{2,l-1}(\Omega)$ for all $l \ge 2$.  
  \label{thm:sobolevreg}
\end{theorem}
\begin{proof}
  In particular this is a simplified version of theorem 6.6.1 in \cite{kmr97}.
  Just note that $\beta = l-1$ fulfills the condition  
  \begin{equation*}
    -\pi/\alpha_i < l-1-\beta < \pi/\alpha_i, \quad \textnormal{for all } i=1,\ldots,d.
  \end{equation*}
  since $0<\alpha_i<2\pi$ for all $i=1,\ldots,d$.
\end{proof}

Having established the regularity result for boundary value problems with
respect to the weighted Sobolev spaces we can use a bootstrapping technique
to deduce the regularity of the eigenfunctions.

\begin{theorem}
  Let $u \in H^1_0(\Omega)$ be an eigenfunction for the Poisson eigenvalue
  problem of equation \eqref{eq:poisson}. Then
  \begin{equation*}
    u \in \hat V^l_{2,l-1}, \quad \textnormal{for all } l \ge 0.
  \end{equation*}
  \label{thm:eigensobolev}
\end{theorem}
\begin{proof}
  The theorem is proven by induction. Given the existence of $u$ as an
  eigenfunction we basically use the fact that $u$ is the solution of
  the boundary value problem with right hand side $\lambda u$. Hence with
  every application of Theorem \ref{thm:sobolevreg} we gain smoothness
  for the eigenfunction $u$.

  We start the induction by noting that the eigenfunction $u \in L^2(\Omega)$.
  From Definition \ref{def:weightedsobolev} it follows directly that
  \begin{equation*}
    \|u\|_{V^0_{2,1}} \lesssim \|u\|_{L^2(\Omega)}.
  \end{equation*}
  Applying Theorem \ref{thm:sobolevreg} assures that the solution $u \in \hat V^{2}_{2,1}$.

  Using the above fact as a starting point for an induction we will show that
  $u \in \hat V^{l}_{2,l-1}$ for all $l=2,4,\ldots$. Suppose that $u \in \hat V^{l}_{2,l-1}$
  for $l$ even. Then by Definition \ref{def:weightedsobolev} of the weighted Sobolev spaces it
  is obvious that $u \in V^{l}_{2,l+1}$. Using Theorem \ref{thm:sobolevreg} it follows that
  \begin{equation*}
    \| u \|_{\hat V^{l+2}_{2,l+1}} \lesssim \| u \|_{V^{l}_{2,l+1}}
  \end{equation*}
  and therefore $u \in \hat V^{l+2}_{2,l+1}$. 

  Hence $u \in \hat V^{l}_{2,l-1}$ for all $l=2,4,\ldots$. Noting that
  \begin{equation*}
    \| u \|_{V^{l-1}_{2,l-2}} \le \| u \|_{V^{l}_{2,l-1}}
  \end{equation*}
  by Definition \ref{def:weightedsobolev} finishes the proof.
\end{proof}

In conclusion we have established arbitrarily high Sobolev regularity for the
eigenfunctions with respect to the appropriate weights. It remains to show that
this also implies regularity in the sense of Besov spaces. In order to show 
that, we use the norm equivalence between certain Besov
norms and the weighted discrete norms of the coefficients of a corresponding wavelet
expansion. We will follow the line of proof that was also used to prove
Besov regularity in \cite{dahlke99a,dd97} for the corresponding
boundary value problem. 

For that purpose we use two-dimensional wavelets constructed from 
univariate orthonormal Daubechies wavelets \cite{daubechies92}.

Then for the set of supports
\begin{equation*}
 \quad I = 2^{-j}k + 2^{-j}[0,1]^2,
  \quad k \in \mathbb Z^2, \; j \in \mathbb Z,
\end{equation*}
the
functions
\begin{equation*}
  \eta_I := \eta_{j,k} := 2^{j}\eta(2^j \cdot -k), \quad \eta \in \Psi,
\end{equation*}
form an orthonormal basis in $L^2(\mathbb R^2)$. 
Here $\Psi$ is a set of three 
wavelets, since the space dimension is two. 

For sufficiently regular wavelets there is a norm equivalence between 
the Besov norm and the discrete norm of the wavelet expansion: A function
$f \in B^{\alpha}_{\tau}(L^{\tau}(\mathbb R^2))$ for $1/\tau = \alpha/2 + 1/2$
if and only if
\begin{equation*}
  \| P_0(f)\|_{L^2(\mathbb R^2)} + 
    \left( \sum_{\eta \in \Psi} \sum_{I \in \mathcal D^{+}} 
      | \langle f, \eta_I \rangle |^{\tau} \right)^{1/\tau} < \infty,
\end{equation*}
where $\mathcal D^{+}$ denotes the set of all dyadic cubes of measure $<1$ and
$P_0$ is a projector onto a suitable subspace of $L^{2}(\mathbb R^2)$.

From classical
regularity analysis of Grisvard \cite{grisvard92} we know that the eigenfunction $u \in H^{3/2}(\Omega)$.
In order to apply the above norm equivalence, we first extend $u$ to the 
whole space $\mathbb R^2$ by a Whitney extension. We will denote the corresponding extended function also by $u$; 
it  also has Sobolev regularity $3/2$, i.e. $u \in H^{3/2}(\mathbb R^2)$.

Now we proceed along the same lines as in \cite{dahlke99a}. The approximation in the
interior of $\Omega$ and on the coarsest scale $P_0(u)$ do not restrict the
Besov regularity. All what remains is to estimate the wavelet coefficients in 
the vicinity of the vertices. For that purpose we introduce the distance from 
a fixed vertex $x^{(i)}$ as
\begin{equation*}
  \delta_I := \inf_{x \in Q(I)} \|x - x^{(i)}\|_{\mathbb R^2},
\end{equation*}
where $Q(I)$ is a cube completely containing the support of $\eta_I$. Using the
approximation order of wavelets, see e.g. \cite{cohen03} we get
\begin{equation*}
  | \langle u, \eta_I \rangle | \lesssim 2^{-n j} |u|_{W^{n}_2(L^2(Q(I)))}.
\end{equation*}
If $\delta_I >0$ we can further estimate for each $\alpha$ such that $|\alpha| = n$
\begin{align*}
  \| \partial^{\alpha}u\|_{L^2(Q(I))}^2 &= \int_{Q(I)}|\partial^{\alpha} u|^2\; \mathrm d x
  \le \int_{Q(I)} \left( \frac{\rho^{n-1}}{\delta_{I}^{n-1}}\right)^2 |\partial^{\alpha} u|^2\; \mathrm d x \\
  & \lesssim (\delta_{I}^{1-n})^2 \|u\|_{V^{n}_{2,n-1}}^2,
\end{align*}
hence $\|u\|_{W^n(L^2(Q(I)))} \lesssim \delta_I^{1-n}$.

Next we sum up the contributions level by level. We introduce the set
of indices corresponding to level $j$ by
\begin{equation*}
  \Lambda_j:= \{ (I,\eta) , |I| = 2^{-2j}\}
\end{equation*}
and for each level appropriate layers
\begin{equation*}
  \Lambda_{j,k} := \{(I,\eta) \in \Lambda_j, \; k2^{-j}\le \delta_I < (k+1)2^{-j}\}.
\end{equation*}
The wavelets in the vicinity of the vertices $x^{(i)}$ have to be treated 
separately and we restrict ourselves to the set $\Lambda_{j,k}$ for $k\ge k_1$. 
Then it follows that
\begin{align*}
  \sum_{k=k_1}^{\infty} \sum_{(I,\eta)\in \Lambda_{j,k}} |\langle u, \eta_I\rangle|^{\tau}
  & \lesssim \sum_{k=k_1}^{\infty} \sum_{(I,\eta)\in \Lambda_{j,k}} 2^{-jn\tau} \delta_{I}^{(1-n)\tau} \\
  & \lesssim 2^{-j\tau} \sum_{k=k_1}^{\infty}k^{1+(1-n)\tau},
\end{align*}
since the cardinality of $\Lambda_{j,k}$ is proportional to $k$. Choosing
$n$ large enough ensures that the sum involving $k$ is finite. The summation
over the refinement levels $j$ then amounts to sum up a geometric series.

Now the coefficients in the vicinity of the vertices can be estimated as in
\cite{dd97}. Then finally we arrive at a regularity result in terms of
Besov spaces.

\begin{theorem}
  Let $u \in H^1_0(\Omega)$ be an eigenfunction for the Poisson eigenvalue
  problem of equation \eqref{eq:poisson}. Then
  \begin{equation*}
    u \in B^{\alpha}_{\tau}(L^{\tau}(\Omega)), \quad \textnormal{for all }
    \alpha \ge 0, \quad \textnormal{where } 1/\tau = \alpha/2 + 1/2.
  \end{equation*}
  \label{thm:besovreg}
\end{theorem}

\subsection{Numerical example}

The last paragraph showed that the eigenfunctions of our model problem have arbitrarily high
regularity in the sense of Besov spaces. However from classical 
regularity theory, the regularity  in terms of Sobolev spaces is 
restricted by the biggest inner angle. In particular for the
L-shaped domain the lowest eigenfunction can only be shown to be in
$H^s$ for $s<5/3$. This means that even for piecewise
linear hat functions the convergence rate of for uniform refinement
will be less than $N^{-1/3}$ in the $H^1$ norm, where $N$ is the number
of degrees of freedom.

In contrast to this, usage of a piecewise linear wavelet bases with two vanishing
moments in the adaptive algorithm PPINVIT will lead to an optimal convergence
rate of the eigenfunction in the $H^1$-norm with complexity $N^{-1/2}$ \cite{drsz08}.
Therefore, adaptive solution of the Poisson eigenvalue problem is superior to uniform grid refinement. 
To conclude this paper, we will demonstrate this for the test example of the Poisson 
eigenvalue problem on the L-shaped domain $\Omega=(-1,1)^2 \setminus [0,1]^2$.

The implementation of our eigenvalue solver is based on the adaptive wavelet code
described in \cite{vorloeper09}. The wavelet basis is constructed along the lines 
of \cite{ds99} which admits
a diagonal preconditioner. Moreover the set of active basis functions
is limited to the case where the successor of each wavelet is also contained
in the set. This results in the exact evaluation of an operator application
for a given index set; hence the Rayleigh quotient can be calculated 
exactly.

As an algorithm we used one PPINVIT\_STEP followed by a Galerkin eigenvalue
solution on the fixed index set with an appropriate accuracy followed
by a coarsening of the iterands. It turns out
that the constants from Sections 3 and 4 are too pessimistic and one can use much
smaller values.

%
\begin{figure}
  \begin{center}
    \includegraphics[width = 0.8 \textwidth]{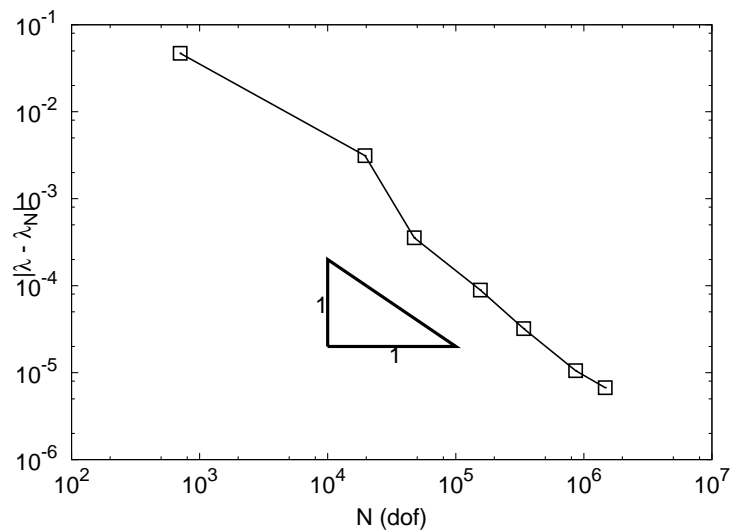}
  \end{center}
  \caption{\label{fig:convergence} Convergence of the Rayleigh quotient for
  the adaptive algorithm with respect to the degrees of freedom.}
\end{figure}
In figure \ref{fig:convergence} the error in the Rayleigh quotient in the smallest
eigenvalue is shown. As a reference value we used $\lambda_1 = 9.639723844$, see \cite{cg08}.
It can be seen that the error decreases like $N^{-1}$, which is as expected
twice as high as the rate for the corresponding eigenfunction.

Also of interest is the structure of the chosen wavelets which is shown 
in figure \ref{fig:index}. The plot shows the center of the active ansatz
functions during the sixth step for two different zoom levels. There one
can see the self similarity in the two scales.
\begin{figure}
  \begin{center}
    \includegraphics[width = 0.45 \textwidth]{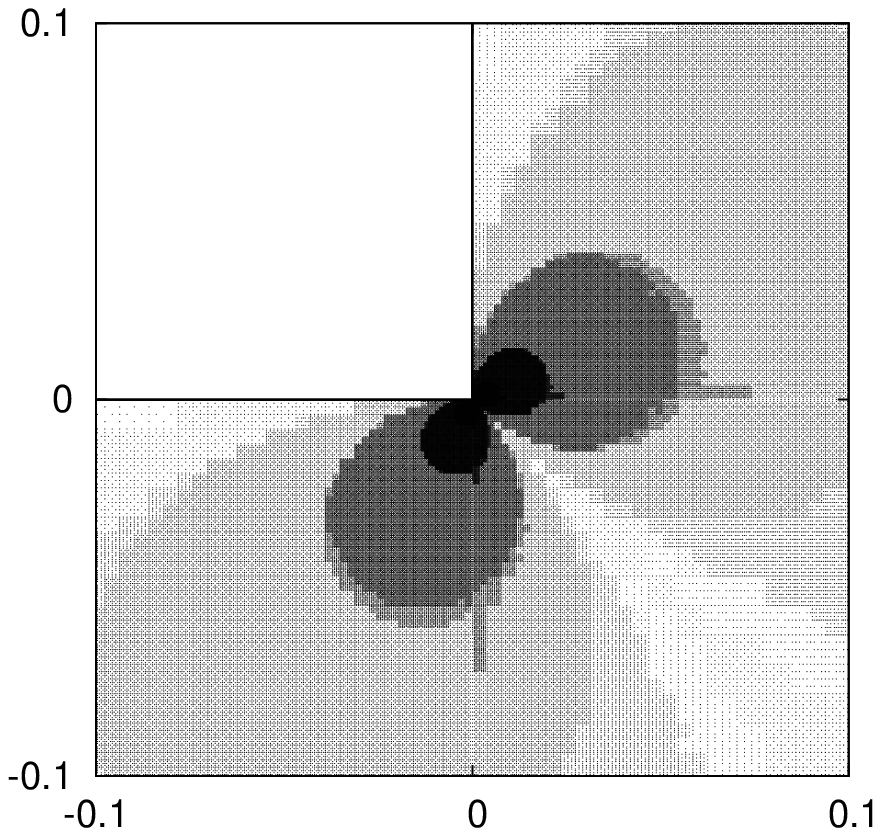}
    \hfill
    \includegraphics[width = 0.45 \textwidth]{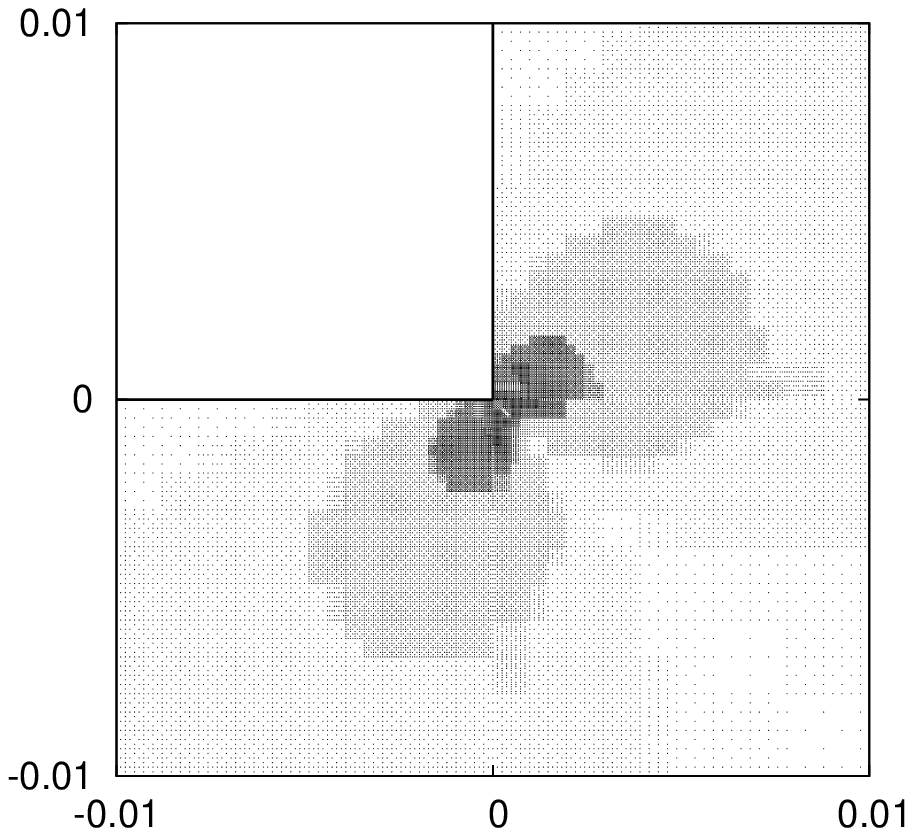}
  \end{center}
  \caption{\label{fig:index} Plot of the active wavelets
  in the sixth step for two different zoom levels. Each dot corresponds to a center of support 
  of an active wavelet.}
\end{figure}

\section*{Acknowledgment}

The authors would like to thank Wolfgang Dahmen, Klaus Neymeyr and Christoph Schwab for
discussions. We are also very thankful to J\"urgen Vorloeper for 
providing the basis of the adaptive wavelet code.

\end{document}